\documentclass[10pt]{article}
\usepackage{amsmath}
\usepackage{latexsym}
\usepackage{amsfonts}
\usepackage{amssymb}
\usepackage{amscd}

\usepackage{amsthm}      
\usepackage[all]{xy}  
\usepackage{remreset}

\def\Bbb{\mathbb}
\def\bcp{\mathbb C \mathbb P}
\def\CC{\mathbb C}
\def\eea{\end{eqnarray*}}

\newtheorem{main}{Theorem}

\newtheorem{defn}{Definition}
\newtheorem{thm}{Theorem}
\newtheorem{prop}[thm]{Proposition}
\newtheorem{cor}[thm]{Corollary}
\newtheorem{lem}[thm]{Lemma}

\def\ro{\stackrel{\circ}{r}}
\def\wt{\widetilde}

\newenvironment{rmk}{\mbox{ }\\{\bf  Remark}\mbox{ }}{
\hfill $\Box$\mbox{}\bigskip}
\def\ZZ{{\mathbb Z}}
\def\RR{{\mathbb R}}

\begin{document}
\def\NN{\mathbb N}
\def\ZZ{\mathbb Z}
\def\QQ{\mathbb Q}
\def\RR{\mathbb R}
\def\CC{\mathbb C}
\def\SS{\mathbb S}
\def\PP{\mathbb P}
\def\VV{\mathbb V}
\def\ro{\mathring{r}}
\def\ve{\varepsilon}
\def\bcp{\mathbb C \mathbb P}
\def\cpb{\overline{\mathbb C \mathbb P^2}}
\def\KK{\cal K}
\def\EE{\cal E}
\def\LL{\cal L}
\def\OO{{\cal O}} 
\def\XX{\cal X}
\def\pp{\cal P}
\def\mm{\cal M}
\def\cc{\cal C}
\def\yy{{\cal Y}}
\def\zz{\cal Z}
\def\eps{\epsilon}
\def\dd{\Delta}
\def\ss{\Sigma}
\def\T{\Theta}

\def\acts{\curvearrowright}

\title{Smooth Structures and Normalized Ricci Flows on Non-Simply Connected Four-Manifolds}

\author{Masashi Ishida and Ioana {\c S}uvaina}


\maketitle

\begin{abstract}
A solution to the normalized Ricci flow is called non-singular if it exists for all time with uniformly bounded sectional curvature. By using the techniques developed by the present authors \cite{ism, sio}, we study the existence or non-existence of non-singular solutions of the normalized Ricci flow on $4-$manifolds with non-trivial fundamental group and the relation with the smooth structures. For example, we prove that, for any finite cyclic group ${\mathbb Z}_{d}$, where $d>1$, there exists a compact topological $4-$manifold $X$ with fundamental group ${\mathbb Z}_{d}$, which admits at least one smooth structure for which non-singular solutions of the normalized Ricci flow exist, but also admits infinitely many distinct smooth structures for which {\it no} non-singular solution of the normalized Ricci flow exists. 

Related non-existence results on non-singular solutions are also proved. Among others, we show that there are no non-singular $\ZZ_d-$equivariant solutions to the normalized Ricci flow on appropriate  connected sums of $\bcp ^2$s and $\cpb $s ($d>1$).
\end{abstract}


\section{Introduction}

The normalized Ricci flow on a closed oriented Riemannian manifold $X$ of dimension $n \geq 3$ is the following evolution equation:
\begin{eqnarray*}\label{Ricci}
 \frac{\partial }{\partial t}{g}=-2{Ric}_{g} + \frac{2}{n}\Big(\frac{{\int}_{X} {s}_{g} d{\mu}_{g}}{{\int}_{X}d{\mu}_{g}} \Big) {g}, 
\end{eqnarray*}
where ${Ric}_{g}$ is the Ricci curvature of the evolving Riemannian metric $g$, ${s}_{g}$ denotes the scalar curvature of $g$ and $d{\mu}_{g}$ is the volume measure with respect to $g$. Recall that an one-parameter family of metrics $\{g(t)\}$, where $t \in [0, T)$ for some $0<T\leq \infty$, is called a solution to the normalized Ricci flow if it satisfies the above equation for all $x \in X$ and $t \in [0, T)$. A solution $\{g(t)\}$ on a time interval $[0, T)$ is said to be maximal if it can not be extended past time $T$. Let us recall the following definition, which was first introduced and studied by Hamilton \cite{ha-0, c-c}: 
\begin{defn}\label{non-sin}
A maximal solution $\{g(t)\}$, $t \in [0, T)$, to the normalized Ricci flow on $X$ is called non-singular if $T=\infty$ and the Riemannian curvature tensor $Rm_{g(t)}$ of $g(t)$ satisfies 
\begin{eqnarray*}
\sup_{X \times [0, T)}|Rm_{g(t)}| < \infty. 
\end{eqnarray*}
\end{defn} 
Hamilton \cite{ha-0} classified non-singular solutions to the normalized Ricci flow on $3-$manifolds. Though there are many results on the non-singular solutions in higher dimensions $n \geq 4$ (cf. \cite{c-c}), the existence or non-existence of non-singular solutions to the normalized Ricci flow is still mysterious in general. In this article, we would like to study the case of $n=4$. One of special properties in dimension 4 is the existence of exotic smooth structures (cf. \cite{don, free}). By using Seiberg-Witten monopole equations \cite{w}, the first author \cite{ism} pointed out that the existence or non-existence of non-singular solutions to the normalized Ricci flow on $4-$manifolds  depends on the choice of smooth structure. However, all known results were for the simply connected case. The main purpose of this article is to explore the case of non-simply connected $4-$manifolds by using the techniques developed by the present authors \cite{ism, sio}.  \par
To state our main results precisely, let us briefly recall the definition of the Yamabe invariant of any closed oriented Riemannian manifold $X$ of dimension $n\geq 3$. By the affirmative solution of the Yamabe problem \cite{yam}, which is due to Trudinger, Aubin, and Schoen \cite{aubyam,rick,trud}, every conformal class on any smooth compact manifold contains a Riemannian metric of constant scalar curvature. For each conformal class $[g]=\{ vg ~|~v: X\to {\Bbb R}^+\}$, we can consider the following associated number
\begin{eqnarray*}
Y_{[g]} = \inf_{h \in [g]}  \frac{\int_X 
s_{{h}}~d\mu_{{h}}}{\left(\int_X 
d\mu_{{h}}\right)^{\frac{n-2}{n}}}. 
\end{eqnarray*}
This is called Yamabe constant of the conformal class $[g]$. Kobayashi \cite{kob} and Schoen \cite{sch} independently introduced the following invariant of $X$:
\begin{eqnarray*}
{\mathcal Y}(X) = \sup_{ [g] \in \mathcal{C}}Y_{[g]}, 
\end{eqnarray*}
where $\mathcal{C}$ is the set of all conformal classes on $X$. This invariant is called Yamabe invariant of $X$. \par
It was pointed out in \cite{ism} that for any closed oriented smooth $4-$manifold $X$ with ${\mathcal Y}(X)<0$, the existence of the non-singular solution of the normalized Ricci flow on $X$ forces the following topological constraint on the Euler characteristic $\chi(X)$ and signature $\tau(X)$ of $X$:
\begin{eqnarray*}\label{FZZ}
2 \chi(X) > 3|\tau(X)|. 
\end{eqnarray*}
In this article, let us call this the strict generalized Hitchin-Thorpe Inequality. Under these notations, the first main result of the present article can be stated as follows: 
\begin{main}\label{main-A}
For any finite cyclic group ${\mathbb Z}_{d}:={\mathbb Z}/d{\mathbb Z}$, where $d>1$, there exist infinitely many pairs 
\begin{eqnarray*} 
\Big( X_i ,\{Z_{i,j} \}_{j \in \NN} \Big)_{i \in {\mathbb N}}
\end{eqnarray*}
of compact, oriented, smooth $4-$manifolds with fundamental group ${\mathbb Z}_{d}$ and satisfying the following properties:
\begin{itemize}
\item [ 1. ] For any fixed $i$, any two manifolds in $\{X_i,Z_{i,1},Z_{i,2},\cdots\}$ are
 homeomorphic, but no two are diffeomorphic to each other; 
\item [ 2. ] $X_i$ satisfies  ${\mathcal Y}(X_i)<0$ and there exist non-singular solutions to the the normalized Ricci flow on $X_i$. In particular, $X_i$ satisfies the strict generalized Hitchin-Thorpe Inequality, i.e., $2 \chi(X_i) > 3|\tau(X_i)|$. 
\item [ 3. ] $Z_{i,j}$ also satisfies ${\mathcal Y}(Z_{i,j}) <0$ and $2\chi(Z_{i,j})>3|\tau(Z_{i,j})|$, but there are no non-singular solutions to the normalized Ricci flow on $Z_{i,j}$. 
\end{itemize} 

 Moreover, their universal covers, $\widetilde{X_i}$ and $ \widetilde{Z_{i,j}},$ satisfy the following properties:
\begin{itemize}
\item [ 4. ] $\widetilde{Z_{i,j}}$ is diffeomorphic to $ n\bcp^2 \# m\overline{\bcp}^2$, where $n=b_2^+(\widetilde{X_i})$ and $m=b_2^-(\widetilde{X_i})$.  
\item [ 5. ] $\widetilde{X_i} $ and $\widetilde{Z_{i,j}}$ are not diffeomorphic, but they become diffeomorphic after taking the connected sum with one copy of $\bcp^2$.
\end{itemize} 

\end{main}

One of the ingredients needed in the proof of  Theorem \ref{main-A} is an existence theorem of complex surfaces of general type with finite cyclic fundamental group, which was proved by the second author \cite{sio}.  \par
For suitable infinitely many integers $n,m$, on connected sums $n {\mathbb C}{P}^2 \# m \overline{{\mathbb C}{P}^2}$, it was pointed out in \cite{ism} that there are infinitely many distinct smooth structure for which no non-singular solutions to the normalized Ricci flow exists. In particular, such smooth structures have negative Yamabe invariants. On the other hand, for the standard smooth structure on $n {\mathbb C}{P}^2 \# m \overline{{\mathbb C}{P}^2}$, the Yamabe invariant must be positive. In this case, a general obstruction like $2 \chi > 3|\tau|$ to the existence of non-singular solution to the normalized Ricci flow is still unknown. Hence, in the case of ${\mathcal Y}>0$, the situation is quite different from the case of ${\mathcal Y}<0$.  In what follows we endow our manifolds with symmetries induced by a group action and we study the equivariant situation. We have the following definition

\begin{defn}
Let $G$ be a group acting smoothly $\rho: G\acts X$ on the closed smooth $4-$manifold $X$. Let $\{g(t)\}$, $t \in [0, T), 0<T \leq \infty$ be a solution of the normalized Ricci flow on $X$. We shall say that  $\{g(t)\}$ is a $G-$equivariant solution of the Ricci flow if the metric $g(t)$ is $G$-invariant for all $t \in [0, T)$. 
\end{defn}
We study the special case when $G$ is a finite group acting freely on $X.$ In this situation, it is sufficient to stat with a $G-$invariant initial metric $g_0.$ The second main result of this article can be stated as follows:
\begin{main}\label{main-B}
Let  $d\geq 2$ be an integer number. For any small $\delta >0$, there is a positive constant $C(\delta)$ such that for any  integer lattice point $(n,m)$ satisfying $d/n,d/m$ and $0<n < (6-\delta)m - C(\delta)$, there exist infinitely many, distinct, free, smooth ${\mathbb Z}_{d}$-actions on the smooth 4-manifold:
\begin{eqnarray*}
X:=(2m-1) {\mathbb C}{P}^2 \# (10m-n-1) \overline{{\mathbb C}{P}^2}, i.e. (2\chi+3\tau)(X)=n, \frac{\chi+\tau}4(X)=m.
\end{eqnarray*}
 Moreover, on $X$ there is  no non-singular $\ZZ_d-$equivariant solution to the normalized Ricci flow, 
 for any of the above $\ZZ_d-$actions.
 \end{main}

As the quotient manifolds are smooth and have the same topological invariants, they are homeomorphic according to Theorem \ref{ha-kr}. The actions of $\ZZ_d$ are distinct in the sense that the quotient manifolds have {\em distinct smooth structures.}

We also prove related results on non-existence of non-singular solutions to the normalized Ricci flow in Section \ref{related} below. 
 
\bigskip
\noindent
{\bf Acknowledgments.} 
We would like to express our deep gratitude to Claude LeBrun for his warm encouragements. The first author is partially supported by the Grant-in-Aid for Scientific Research (C), Japan Society for the Promotion of Science, No. 20540090. The second author would like to thank the Institut des Hautes {\'E}tudes Scientifiques for its warm hospitality.

\section{Obstructions to the existence of non-singular solutions to the normalized Ricci flow}

Let us introduce the following definition: 
\begin{defn}[\cite{ism}]\label{bs}
A maximal solution $\{g(t)\}$, $t \in [0, T)$, to the normalized Ricci flow  on $X$ is called quasi-non-singular if $T=\infty$ and the scalar curvature $s_{g(t)}$ of $g(t)$ satisfies 
\begin{eqnarray*}
\sup_{X \times [0, T)}|{s}_{g(t)}| < \infty. 
\end{eqnarray*}
\end{defn}
Notice that any non-singular solution is quasi-non-singular, but the converse is not true in general. \par
By using the Seiberg-Witten monopole equations, the following obstruction to the existence of quasi-non-singular solutions of the normalized Ricci flow on $4-$manifolds was proved: 
\begin{thm}[\cite{ism}]\label{ricci-ob-1}
Let $X$ be a closed almost-complex $4-$manifold with ${b}^+(X) \geq 2$ and $2\chi(X) + 3 \tau(X)>0$. Assume that $X$ has a non-trivial integer valued Seiberg-Witten invariant $SW_{X}(\Gamma_{X}) \not=0$, where $\Gamma_{X}$ is the spin${}^c$-structure compatible with the almost-complex structure. Let $N$ be a closed, oriented, smooth, $ 4-$manifold with $b^{+}(N)=0$. Then, there are no quasi-non-singular solutions to the normalized Ricci flow in the sense of Definition \ref{bs} on the connected sum $M:=X \# N$ if the following holds:
\begin{eqnarray}\label{ob-N-Ricci-1}
12b_{1}(N) + 3{b}^{-}(N) > 2\chi(X) + 3 \tau(X). 
\end{eqnarray}
In particular, under this condition, there are no non-singular solutions to the normalized Ricci flow in the sense of Definition \ref{non-sin}. 
\end{thm}
In what follows, we  refine this theorem, see Theorem \ref{ricci-ob-3} below. To prove our result we need to recall the following
\begin{lem}[\cite{fz-1, ism}]\label{FZZ-prop}
Let $X$ be a closed oriented Riemannian $n$-manifold and assume that there is a long time solution $\{g(t)\}$, $t \in [0, \infty)$, to the normalized Ricci flow. Assume moreover that the solution satisfies $|{s}_{g(t)}| \leq C$ and 
\begin{eqnarray}\label{mini-bound}
\hat{s}_{g(t)}:=\min_{x \in X}{s}_{g(t)}(x) \leq -c <0,
\end{eqnarray}
where the constants $C$ and $c$ are independent of both $x \in X$ and time $t \in [0, \infty)$. Then, the trace-free part $\stackrel{\circ}{r}_{g(t)}$ of the Ricci curvature satisfies 
\begin{eqnarray*}
{\int}^{\infty}_{0} {\int}_{X} |\stackrel{\circ}{r}_{g(t)}|^2 d{\mu}_{g(t)}dt < \infty. 
\end{eqnarray*}
In particular, when $m \rightarrow \infty$, 
\begin{eqnarray}\label{fzz-ricci-0}
{\int}^{m+1}_{m} {\int}_{X} |\stackrel{\circ}{r}_{g(t)}|^2 d{\mu}_{g(t)}dt \longrightarrow 0. 
\end{eqnarray}

\end{lem}
This lemma first appeared as Lemma 3.1 in \cite{fz-1} for $n=4$ and unit volume solutions. However, it is not hard to see that the same result still holds without such assumptions \cite{ism}. Notice also that Lemma \ref{FZZ-prop} tells us that any quasi-non-singular solution to the normalized Ricci flow satisfying (\ref{mini-bound}) must satisfy (\ref{fzz-ricci-0}). \par
By using Lemma \ref{FZZ-prop}, Seiberg-Witten monopole equations and the estimates of the Yamabe invariant, we can prove
\begin{prop}[\cite{ism}]\label{key-prop}
Let $X$ be a closed almost-complex 4-manifold with ${b}^+(X) \geq 2$ and $2 \chi(X) + 3\tau(X) > 0$. Assume that $X$ has a non-trivial integer valued Seiberg-Witten invariant $SW_{X}(\Gamma_{X}) \not=0$, where $\Gamma_{X}$ is the spin${}^c$-structure compatible with the almost-complex structure. Let $N$ be a closed oriented smooth $4-$manifold with $b^{+}(N)=0$. If a quasi-non-singular solution to the normalized Ricci flow on the connected sum $M:=X \# N$ exists, then
\begin{eqnarray}\label{fzz-ricci-011}
{\int}^{m+1}_{m} {\int}_{M} |\stackrel{\circ}{r}_{g(t)}|^2 d{\mu}_{g(t)}dt \longrightarrow 0
\end{eqnarray}
holds when $m \rightarrow +\infty$. 
\end{prop}
For completeness we sketch the main steps of the proof. Under the assumptions in the proposition, we are able to prove that the Yamabe invariant of the connected sum $M:=X \# N$ satisfies the following bound (\cite{ism}): 
\begin{eqnarray}\label{yama-1}
{\mathcal Y}(M) \leq -4{\pi}\sqrt{2(2 \chi(X) + 3\tau(X))} < 0. 
\end{eqnarray}
On the other hand, it was also showed in \cite{ism} that any solution to the normalized Ricci flow on any closed $4-$manifold $Z$ with negative Yamabe invariant must satisfy
\begin{eqnarray*}
\hat{s}_{g(t)}:=\min_{x \in Z}{s}_{g(t)}(x) \leq  \frac{{\mathcal Y}(Z)}{\sqrt{vol_{g(0)}}} < 0. 
\end{eqnarray*}
This bound and (\ref{yama-1}) tell us that any quasi-non-singular solution to the normalized Ricci flow on the connected sum $M:=X \# N$  must satisfy the bound (\ref{mini-bound}). This observation and Lemma \ref{FZZ-prop} imply Proposition \ref{key-prop}. \par
We are now in the position to prove a slight refinement of Theorem \ref{ricci-ob-1}. Notice that the bound (\ref{ob-N-Ricci-2}) below is slightly stronger than the bound (\ref{ob-N-Ricci-1}). \par
\begin{thm}\label{ricci-ob-3}
Let $X$ be a closed almost-complex $4-$manifold with ${b}^+(X) \geq 2$. Assume that $X$ has a non-trivial integer valued Seiberg-Witten invariant $SW_{X}(\Gamma_{X}) \not=0$, where $\Gamma_{X}$ is the spin${}^c$-structure compatible with the almost-complex structure. Let $N$ be a closed oriented smooth $4-$manifold with $b^{+}(N)=0$. Then the following hold:  
\begin{itemize}
\item [ 1. ] Suppose that $2\chi(X) + 3 \tau(X) \leq 0$. If  $4b_{1}(N) + {b}^{-}(N) \not=0$ holds, then there is no quasi-non-singular solution to the normalized Ricci flow on a connected sum $M:=X \# N$  satisfying (\ref{mini-bound}). In particular, there does not exist a non-singular solution to the normalized Ricci flow.
\item [ 2. ] Suppose that  $2\chi(X) + 3 \tau(X) > 0$. Assume moreover that $N$ is not an integral homology $4-$sphere whose fundamental group has no non-trivial finite quotient. Then, there do not exist quasi-non-singular solutions to the normalized Ricci flow  on the connected sum $M:=X \# N$ if the following holds:
\begin{eqnarray}\label{ob-N-Ricci-2}
12b_{1}(N) + 3{b}^{-}(N) \geq 2\chi(X) + 3 \tau(X). 
\end{eqnarray}
In particular, under this condition, there does not exist a non-singular solution to the normalized Ricci flow. 
\end{itemize}
\end{thm}

\begin{proof}
 It is known \cite{fz-1, ism} that the existence of a quasi-non-singular solution to the normalized Ricci flow on a $4-$manifold $Y$  satisfying (\ref{mini-bound}) forces the bound $2 \chi(Y) \geq 3|\tau(Y)|$. Therefore, if there is a quasi-non-singular solution on $M$ satisfying (\ref{mini-bound}), then $2 \chi(M) \geq 3|\tau(M)|$ must hold. Since $2 \chi(M) + 3\tau(M)=2 \chi(X) + 3\tau(X)-4b_{1}(N) -{b}^{-}(N)$, we have in particular
\begin{eqnarray*}
2 \chi(X) + 3\tau(X) \geq  4b_{1}(N) +{b}^{-}(N).
\end{eqnarray*}
Since we assume that $2 \chi(X) + 3\tau(X) \leq 0$, this bound forces $4b_{1}(N) +{b}^{-}(N)=0$. However, this contradicts $4b_{1}(N) + {b}^{-}(N) \not=0$. Hence, we showed Case 1. \par
In what follows, we  prove the second statement. It is shown in \cite{ism} that any Riemannian metric $g$ on the connected sum $M:=X \# N$ satisfies the following bound: 
\begin{eqnarray}\label{mono-1}
\frac{1}{4{\pi}^2}{\int}_{M}\Big(2|W^{+}_{g}|^2+\frac{{s}^2_{g}}{24}\Big) d{\mu}_{g} \geq \frac{2}{3} c^2_{1}({X}) = \frac{2}{3}\Big( 2 \chi(X) + 3\tau(X)\Big), 
\end{eqnarray}
where $W^{+}_{g}, s_g$ are the self-dual Weyl curvature and the scalar curvature of $g.$  Moreover, the equality holds if and only if the metric $g$ is a K{\"{a}}hler-Einstein metric with negative scalar curvature. \par
Now, suppose moreover that $N$ is not an integral homology 4-sphere whose fundamental group has no non-trivial finite quotient. Then the equality cannot occur in (\ref{mono-1}). We shall prove this as follows. As was already mentioned above, notice that the equality case in (\ref{mono-1}) forces the metric $g$ to be a K{\"{a}}hler-Einstein metric with negative scalar curvature. In particular, this forces  the connected sum $M:=X \# N$ to be a minimal K{\"{a}}hler surface \cite{aubyam, yau}. On the other hand, Theorem 5.4 in \cite{kot} tells us that, if a minimal K{\"{a}}hler surface with $b^+ > 1$ admits the connected sum decomposition $X \# N$, then $N$ must be an integral homology $4-$sphere whose fundamental group has no non-trivial finite quotient. Therefore, we can  conclude that the equality cannot occur in (\ref{mono-1}). Hence we have the  following strict inequality which holds for any Riemannian metric $g$ on $M$: 
\begin{eqnarray}\label{mono-2}
\frac{1}{4{\pi}^2}{\int}_{M}\Big(2|W^{+}_{g}|^2+\frac{{s}^2_{g}}{24}\Big) d{\mu}_{g} > \frac{2}{3}\Big( 2 \chi(X) + 3\tau(X)\Big). 
\end{eqnarray}
Now, suppose that there is a quasi-non-singular solution $\{g(t)\}$ to the normalized Ricci flow on $M$. Then, we have the following Gauss-Bonnet like formula:
\begin{eqnarray}\label{GB}
2\chi(M) + 3\tau(M) = \frac{1}{4{\pi}^2}{\int}_{M}\Big(2|W^{+}_{g(t)}|^2+\frac{{s}^2_{g(t)}}{24}-\frac{|\stackrel{\circ}{r}_{g(t)}|^2}{2} \Big) d{\mu}_{g(t)}.  
\end{eqnarray}
On the other hand, notice that Proposition \ref{key-prop} tells us that the quasi-non-singular solution $\{g(t)\}$ must satisfy (\ref{fzz-ricci-011}). By the formula (\ref{GB}) and (\ref{fzz-ricci-011}), we have the following
\begin{eqnarray*}
2\chi(M) + 3\tau(M) &=& {\int}^{m+1}_{m} \Big(2\chi(M) + 3\tau(M) \Big)dt \\
&=& \frac{1}{4{\pi}^2}{\int}^{m+1}_{m} {\int}_{M}\Big(2|W^{+}_{g(t)}|^2+\frac{{s}^2_{g(t)}}{24}-\frac{|\stackrel{\circ}{r}_{g(t)}|^2}{2} \Big) d{\mu}_{g(t)}dt \\
&\geq & \liminf_{m \longrightarrow \infty}\frac{1}{4{\pi}^2}{\int}^{m+1}_{m} {\int}_{M}\Big(2|W^{+}_{g(t)}|^2+\frac{{s}^2_{g(t)}}{24}-\frac{|\stackrel{\circ}{r}_{g(t)}|^2}{2} \Big) d{\mu}_{g(t)}dt \\
&=& \liminf_{m \longrightarrow \infty}\frac{1}{4{\pi}^2}{\int}^{m+1}_{m} {\int}_{M}\Big(2|W^{+}_{g(t)}|^2+\frac{{s}^2_{g(t)}}{24}\Big) d{\mu}_{g(t)}dt.  
\end{eqnarray*}
This and the bound (\ref{mono-2}) imply that 
\begin{eqnarray*}
2\chi(M) + 3\tau(M) &\geq&  \liminf_{m \longrightarrow \infty}\frac{1}{4{\pi}^2}{\int}^{m+1}_{m} {\int}_{M}\Big(2|W^{+}_{g(t)}|^2+\frac{{s}^2_{g(t)}}{24}\Big) d{\mu}_{g(t)}dt \\ 
 &>& \liminf_{m \longrightarrow \infty}\frac{2}{3} {\int}^{m+1}_{m} \Big( 2 \chi(X) + 3 \tau(X) \Big) dt \\
 &=& \frac{2}{3} \Big( 2 \chi(X) + 3 \tau(X) \Big). 
\end{eqnarray*}
Since we have $2\chi(M) + 3\tau(M) = 2\chi(X) + 3 \tau(X) - 4b_{1}(N)- {b}^{-}(N)$, we get 
\begin{eqnarray*}
 2\chi(X) + 3 \tau(X) - \Big(4b_{1}(N) + {b}^{-}(N)\Big)  >  \frac{2}{3}\Big( 2 \chi(X) + 3 \tau(X) \Big). 
\end{eqnarray*}
Namely, 
\begin{eqnarray*}
12b_{1}(N) + 3{b}_{2}(N) < 2\chi(X) + 3 \tau(X). 
\end{eqnarray*}
By contraposition, we are able to obtain the desired result. 
\end{proof}

As a corollary of Theorem \ref{ricci-ob-3}, we get
\begin{cor}\label{cor}
Let $X$ be a closed symplectic $4-$manifold with $b^+(X) \geq 2$. Let $Y_{d}$ be any rational homology $4-$sphere with $\pi(Y_{d})=\ZZ_{d}$ and consider the following connected sum:
\begin{eqnarray*}
M:=X \# Y_{d} \# k \overline{{\mathbb C}{P}^2}, 
\end{eqnarray*}
\item [ 1. ] Suppose that $2\chi(X) + 3 \tau(X) \leq 0$. For any $k> 0$, there is no quasi-non-singular solution to the normalized Ricci flow on $M$  satisfying (\ref{mini-bound}). In particular, there is no non-singular solution.
\item [ 2. ] Suppose that  $2\chi(X) + 3 \tau(X) > 0$. If 
\begin{eqnarray*}
k \geq \frac{1}{3} \Big(2\chi(X) + 3 \tau(X)\Big),  
\end{eqnarray*}
then, there are no quasi-non-singular solutions to the normalized Ricci flow on $M$. In particular, there is no non-singular solution. 
\end{cor}
Notice that Ue \cite{ue} constructs, for any $d> 1$, a spin rational homology $4-$sphere with $\pi(Y_{d})=\ZZ_{d}$ whose universal cover is $\#(d-1) S^2 \times S^2$. We use Corollary \ref{cor} to prove Theorems \ref{main-A} and \ref{main-B} in Section \ref{sec-4} below. \par
In the case when the manifolds decompose as connected sum of  manifolds with $b^+>0$, the first author has a different type of obstructions
\begin{thm}[\cite{ism}]\label{ricci-ob-2}
For $i= 1,2,3,4$, let $X_{i}$ be a closed almost-complex $4-$manifold whose integer valued Seiberg-Witten invariant satisfies $SW_{X_{i}}(\Gamma_{X_{i}}) \equiv 1 \ (\bmod \ 2)$, where $\Gamma_{X_{i}}$ is the spin${}^c$-structure compatible with the almost-complex structure. Assume that the following conditions are satisfied:
\begin{itemize}
\item $b_{1}(X_{i})=0$, \ $b^{+}(X_{i}) \equiv 3 \ (\bmod \ 4)$, \ $\displaystyle\sum^{4}_{i=1}b^{+}(X_{i}) \equiv 4 \ (\bmod \ 8)$, 
\item $\sum^j_{i=1}\Big( 2\chi(X_{i}) + 3 \tau(X_{i}) \Big) > 0$, where $j=2,3,4$. 
\end{itemize} 
Let $N$ be a closed oriented smooth $4-$manifold with $b^{+}(N)=0$. Then, for $j=2,3,4$, there is no  quasi-non-singular solution to the normalized Ricci flow on the connected sum $M:=\Big(\#^{j}_{i=1}{X}_{i} \Big) \# N$ if the following holds:
\begin{eqnarray*}
12(j-1)+\Big( 12b_{1}(N) + 3{b}^{-}(N) \Big) \geq \sum^j_{i=1}\Big( 2\chi(X_{i}) + 3 \tau(X_{i}) \Big). 
\end{eqnarray*}
In particular, under this condition, there does not exist a non-singular solution to the normalized Ricci flow on $M$. 
\end{thm}

\section{Complex surfaces of general type with finite cyclic fundamental group}

We begin this section with a short discussion on the homeomorphism criteria for non-simply connected manifolds.
A remarkable result of Freedman \cite{free} in conjunction
 with results of Donaldson tells us that
 smooth, compact, simply connected, oriented $4$-manifolds are
 classified by their numerical invariants: Euler characteristic
 $\chi$, signature $\tau$ and Stiefel-Whitney class $w_2$. Maybe
 not as well-known are the results involving the classification
 of  non-simply connected  $4$-manifolds (\cite{ha-kr},
 Theorem C):

\begin{thm}[\cite{ha-kr}]\label{ha-kr}
Let $M$ be a smooth, closed, oriented, $4$-manifold with finite
 cyclic fundamental group. Then $M$ is classified up to
 homeomorphism by the fundamental group, the intersection form
 on 
  $H_2(M,\ZZ)/Tors$ and the $w_2$-type. Moreover, any isometry of
 the intersection form can be realized by a homeomorphism.
\end{thm}

In contrast with the simply connected manifolds, there are three 
 $w_2$-types that can be exhibited: (I) $w_2(\widetilde M) \neq 0$,
 (II) $w_2(M)=0$, and (III) $w_2(\widetilde M)=0,$ but $w_2(M)\neq 0,$ where $\wt M$ is the universal cover of $M.$

Using Donaldson's and Minkowski-Hasse's classification of the
 intersection form we can reformulate this theorem on an easier
 form: 

{\bf Equivalently}:
A smooth, closed, oriented $4$-manifold with finite cyclic
 fundamental group and indefinite intersection form is
 classified up to homeomorphism by the fundamental group, the
 numbers $b^\pm,$ the parity of the intersection form and
 the $w_2$-type.

Although not any group can be exhibited as the fundamental group of a complex surface, in the case of finite cyclic groups there are constructions which yield such complex manifolds. Starting from iterated cyclic covers of  $\bcp^1\times\bcp^1$ the second author \cite{sio} introduced a method of constructing such manifolds for which the properties are well understood:

\begin{prop}[\cite{sio}]\label{sio-thm}
Given any integer $d\geq2,$ there are infinitely many complex surfaces of general type, $\{X_i\}_{i\in \NN},$ with ample canonical line bundle, such that their fundamental groups $\pi_1(X_i)=\ZZ_d,$ they have $w_2-type ~I,$ odd intersection form and satisfy $c_1^2(X_i)<5\chi_h(X_i).$ Moreover, the universal cover of $X_i$, $\wt{X_i},$ is  not diffeomorphic to connected sums of $\bcp^2$'s and $\cpb$'s, but after taking the connected sum with one copy of $\bcp^2$ the manifold $\wt{X_i}\#\bcp^2$ is diffeomorphic to $n\bcp^2\#m\cpb$ for appropriate $n,m.$
\end{prop}
 We call the manifolds satisfying the above diffeomorphism conditions {\em almost completely decomposable}.
We use $\chi_h$ to denote the Todd genus or the holomorphic Euler characteristic. This is a topological invariant and it can be computed in terms of the other topological invariants as $\chi_h=\frac14(\chi+\tau).$ As the topological invariants of the manifolds constructed increase rapidly, we assume that $X_i$ are ordered such that $\chi_h(X_i)$ is an increasing sequence.

The ampleness of the canonical line bundle is necessary for the first Chern class to be definite $c_1(X_i)<0.$ In this situation, a result of Cao \cite{c, c-c} proves the existence of non-singular solutions to the normalized Ricci flow. We  recall the following version of Cao's result which appears in \cite{c-c}. 
\begin{thm}[\cite{c, c-c}]\label{cao-K}
Let $M$ be a compact K{\"{a}}hler manifold with definite first Chern class ${c}_{1}(M)$. If ${c}_{1}(M)=0$, then for any initial K{\"{a}}hler metric $g_{0}$, the solution to the normalized Ricci flow exists for all time and converges to a Ricci-flat metric as $t \rightarrow \infty$. If ${c}_{1}(M) < 0$ and the initial metric $g_0$ is chosen to represent the opposite of the first Chern class, then the solution to the normalized Ricci flow exists for all time and converges to an Einstein metric of negative scalar curvature as $t \rightarrow \infty$. If ${c}_{1}(M) > 0$ and the initial metric $g_0$ is chosen to represent the first Chern class, then the solution to the normalized Ricci flow exists for all time. 
\end{thm}
Notice that, in case where ${c}_{1}(M) = 0$ or ${c}_{1}(M) < 0$, the solution is actually non-singular in the sense of Definition \ref{non-sin}. See also \cite{c-c}. \par

\section{Proofs of Theorems \ref{main-A}  and \ref{main-B}}\label{sec-4}

To be able to give the proofs of the main theorems we need one more ingredient, which is the abundance of symplectic $4-$manifolds with nice given properties. A suitable construction due to Braungardt and Kotschick  \cite{b-k} exhibits such manifolds:

\begin{prop}[\cite{b-k}]\label{geogr-sympl}
For every $\epsilon >0,$ there is a constant $c(\epsilon)>0$ such
 that every integer lattice point $(x,y)$ in the first quadrant satisfying
 $$ y\leq (9-\epsilon)x- c(\epsilon)$$
 is realized by the Chern invariants $(\chi_h, c_1^2)$ of
 infinitely many, pairwise non-diffeomorphic, simply connected,
 minimal, symplectic manifolds, all of which are almost completely decomposable.
 \end{prop}

The smooth structures constructed in the above proposition are distinguished by the Seiberg-Witten basic classes.
We are now ready to give the proofs of the main theorems:
\begin{proof}[Proof of Theorem \ref{main-A}]
Let $d\geq2$ be an arbitrary integer.
Proposition \ref{sio-thm} provides us with an infinite family of manifolds $\{X_i\}$ with fundamental group $\pi_1(X_i)=\ZZ_d, $ all of which admit non-singular solutions to the normalized Ricci flow by Theorem \ref{cao-K}. Moreover, as they are complex surfaces of general type, their Yamabe invariant $\yy (X_i)<0,$ by a result of LeBrun \cite{L-1}. \par
By Proposition \ref{geogr-sympl} there exists an integer $n_0>0$ such that any integer lattice point $(x,y)$ in the first quadrant, satisfying $x>n_0, y\leq8.5 x$ is represented by infinitely many  minimal symplectic manifolds with the required properties. As the numerical invariants of the sequence $\{X_i\}$ are increasing, we can assume that all these manifolds satisfy $\chi_h(X_i)>n_0.$
Hence, for any $x_i=\chi_h(X_i)$, Proposition \ref{geogr-sympl} tells us that there exist infinitely many homeomorphic, non-diffeomorphic minimal symplectic manifolds $Z'_{i,j}$ such that $\chi_h(Z'_{i,j})=\chi_h(X_i)$ and $9x_i>c_1^2(Z'_{i,j}) >8 x_i.$
Let 
$$
Z_{i,j}=Z'_{i,j}\#Y_d\#k_i\cpb, ~\text{where} ~k_i=c_1^2(Z'_{i,j})-c_1^2(X_i),
$$
and where $Y_d$ is a rational homology sphere with $\pi_1(Y_d)=\ZZ_d.$ Hence the topological invariants of $Z_{i,j}$ are $\pi_1(Z_{i,j})=\pi_1(Y_d)=\ZZ_d=\pi_1(X_i), c_1^2(Z_{i,j})=c_1^2(X_i), \chi_h(Z_{i,j})=\chi_h(X_i).$ As $Z_{i,j}$ is of $w_2-$type $I,$ Theorem \ref{ha-kr} tells us that $X_i$ and $Z_{i,j}$ are all homeomorphic. 
The manifolds $Z_{i,j}$ are non-diffeomorphic as their Seiberg-Witten basic classes \cite{k-m-t} remain distinct after taking the connected sum with $\cpb$ and $Y_d$ . None of them is diffeomorphic to $X_i$ as none of them has no non-singular solutions to the normalized Ricci flow, as we see next.

The number of $\cpb$'s, $k_i,$ can be estimated to be
$$
k_i=c_1^2(Z'_{i,j})-c_1^2(X_i)>8x_i-5x_i=3x_i>3\cdot\frac19c_1^2(Z'_{i,j})=\frac13(2\chi+3\tau)(Z'_{i,j})
$$
Hence, by Corollary \ref{cor}, there are no non-singular solutions to the normalized Ricci flow on any of the $Z_{i,j}.$ The manifolds $Z_{i,j}$ are not symplectic manifolds, but nevertheless they have non-trivial Seiberg-Witten invariant \cite{k-m-t}, so they must have $\yy(Z_{i,j})<0$ (cf. \cite{L-1, ism}). \par
The universal cover of $Z_{i,j}$ is diffeomorphic to the connected sum 
$$dZ'_{i,j}  \#  d k_i \cpb  \#(d-1)(S^2\times S^2)= dZ'_{i,j}\#(d-1)\bcp^2 \#(d-1+d k_i)\cpb.$$ But $Z'_{i,j}$ has the property that after taking the connected sum with one $\bcp^2,$ it decomposes as connected sums of $\bcp ^2$'s and $\cpb $'s. Hence the universal cover of $Z_{i,j}$ is diffeomorphic to $n\bcp^2\#m\cpb$ for appropriate $n$ and $m.$
The universal cover of $X_i$ is a complex manifold of general type, and it is not going to be diffeomorphic to connected sums of $\bcp ^2$'s and $\cpb $'s. But after taking the connected sum with one $\bcp^2$  it is going to decompose, by the construction in Proposition \ref{sio-thm}.
\end{proof}

The ideas used in the above proof can be employed to give us results on equivariant solutions of the normalized Ricci flow on manifolds of the form $a\bcp^2\#b\cpb$ with the canonical smooth structure.
They complement the results of the first author \cite{ism} on exotic smooth structures.

\begin{proof}[Proof of Theorem \ref{main-B}]
Let $\delta >0$ be a small positive number and let $\epsilon =\frac32\delta$.
Proposition \ref{geogr-sympl} tells us that for  this $\epsilon$ there exists an $c(\epsilon)$ such that to any integer lattice point $(x,y)$ satisfying:
 \begin{equation}\label{region}
 0< y\leq (9-\epsilon)x- c(\epsilon)
 \end{equation}
we can associate infinitely many homeomorphic, non-diffeomorphic, simply-connected, almost completely decomposable, minimal symplectic manifolds which have topological invariants $(\chi_h,c_1^2)=(x,y).$  

Let $C(\delta)=\frac{2d}{3}(c(\epsilon)+1).$

Let $n,m$ be positive integer numbers such that $d/n,d/m$ and 
$$n ~~< ~~ (6-\delta)m-C(\delta)$$
Or equivalently:
\begin{align}
\frac{n}{d} &~<~ (6-\delta)\frac{m}{d}- \frac{C(\delta)}{d}\notag\\
\frac{3n}{2d} & ~< ~ (9-\frac32\delta)\frac{m}{d}- \frac{3C(\delta)}{2d} \notag\\
\frac{3n}{2d} & ~< ~(9-\epsilon)\frac{m}{d}-c(\epsilon)-1 \notag\\
\frac{3n}{2d} +1 & ~< ~(9-\epsilon)\frac{m}{d}-c(\epsilon) \notag
\end{align}

Then $[\frac{3n}{2d}] +1$ and $\frac{m}{d}$ satisfy the condition  (\ref{region}), hence there are infinitely many symplectic manifolds $M_i$ such that $c_1^2(M_i)=[\frac{3n}{2d}] +1$ and $\frac{\chi+\tau}{4}(M_i)=\frac{m}{d}$. Moreover the manifolds $M_i$ are homeomorphic, non-diffeomorphic, almost completely decomposable manifolds. The differential structures are distinguished by the Seiberg-Witten basic classes.

Let:
$$X_i=M_i\#Y_d\#k\cpb,~\text{with}~ k=[\frac{3n}{2d}] +1-\frac nd=[\frac n{2d}]+1>\frac13 c_1^2(M_i)$$
Then $\frac{\chi+\tau}{4}(X_i)=\frac md$ and $(2\chi+3\tau)(X_i)=\frac{n}{d}.$ 
The manifolds $X_i$ remain homeomorphic to each other, and  using the formula for the Seiberg-Witten basic classes of the connected sum \cite{k-m-t} we can immediately see that any two manifolds are not diffeomorphic. Their universal cover $\wt{X_i}$ is diffeomorphic to $dM_i\#dk\cpb\#(d-1)(S^2\times S^2)$ and as $M_i$ are almost completely decomposable, this implies that all $\wt{X_i}$ are diffeomorphic to $X=(2m-1)\bcp^2\#(10m-n-1)\cpb$, i.e.  $\frac{\chi+\tau}{4}(X)=m,(2\chi+3\tau)(X)=n.$

We can use Corollary \ref{cor} to conclude that the manifolds $X_i$ don't admit  non-singular solutions of the normalized Ricci flow. This immediately implies the non-existence of the $\ZZ_d$ invariant solutions of the normalized Ricci flow for any of the infinitely many $\ZZ_d$ group actions, corresponding to the fundamental group action on the universal covers $\wt{X_i}=X.$
\end{proof}

For simplicity, we state the theorem for finite cyclic groups, but the results are also true for any finite groups acting freely on the $3-$dimensional sphere, or for direct sums of the above groups. Of course the divisibility condition needs to be changed according to the order of the given group.

Unfortunately, requiring that a metric $g$ is $\ZZ_d-$invariant is a strong condition.
\begin{prop}
Let $G$ be one of the above $\ZZ_d$ actions on $X_{n,m}=(2m-1)\bcp^2\#(10m-n-1)\cpb$ and let $g$ be a $G-$invariant metric. Then $Y_{[g]}<0.$
\end{prop}
\begin{proof}
Let  $G$ and $g$ as in the statement of the proposition. Let 
$$\pi:X_{n,m}\to\hat{X}:=X_{n,m}/G$$
 be the natural projection, where $\hat{X}$ is the smooth $4-$manifold obtained by taking the quotient of $X_{n,m}$ under the free action of $G.$
As $g$ is $G-$invariant, we have an induced metric $\hat g=\pi_*g$ on $\hat{X}.$ But $\hat X$ is one of the manifolds $X_i$ constructed in the previous proof, hence it has non-trivial Seiberg-Witten invariant. This implies \cite{L-1} that ${\mathcal Y}(\hat X)<0$ and of course $Y_{[\hat g]}<0.$ The solution of the Yamabe problem tells us that there exists a metric $g'\in[\hat g]$ with negative constant scalar curvature $s_{g'}\equiv ct<0.$ 

Then $\pi^*(g')$ is a metric of negative constant scalar curvature on $X_{n,m}$ in the conformal class of $g.$ So we can conclude that $Y_{[\pi^*g']}=Y_{[g]}<0.$
\end{proof}

We would like to remark that the constant $c(\epsilon)$ is a large positive number and increases rapidly. As far as the authors know, one of the examples with the smallest topology for which the methods in the above theorem can be used is the following:
\begin{cor}
On $15\bcp^2 \# 77 \overline{\bcp^2}$, there exists an
 involution $\sigma$, acting freely, such that
 $15 \bcp^2 \# 77 \overline{\bcp^2}$ does not admit a $\sigma-$equivariant non-singular solution
of the normalized Ricci flow.
\end{cor}

\begin{proof}
Let $N$ be the double cover of $\bcp^2$ branched
 along a smooth divisor $D$, such that $\OO(D)=\OO_{\bcp^2}(8).$ Then \cite{sio}, $N$ is 
 a simply connected,  almost completely decomposable, surface of general type and its numerical invariants 
 are:
 $$c_2(N)=46,~c_1^2(N)=2,~(~\tau(N)=-30,~b^+=7~ \text{and}~ b^-=37).$$
By Corollary \ref{cor}, the manifold $M=N \# \cpb \# Y_2$
  does not admit non-singular solutions to the normalized Ricci flow. 

Let $\wt{M}$ be the universal cover of $M$. 
Then we have the
 following diffeomorphisms:
$$\wt{M}\cong 2N \# 2\cpb \# (S^2\times S^2)\cong 2N\# \bcp^2
\# 3\cpb  \cong 15\bcp ^2\# 77\cpb.$$

As $M$ does not admit any non-singular solutions to the normalized Ricci flow, this implies that the same is true for $\sigma-$equivariant solutions on $\wt{M}$.
\end{proof}

\section{Related results}\label{related}

As there are no homeomorphism criteria for an arbitrary fundamental group, we can not generalized Theorem \ref{main-A} to other fundamental groups. But we can prove non-existence theorems for a large class of fundamental groups.
\begin{main}\label{non-spin+arb}
Let $G$ be any finitely presented group. For any $\delta>0$ small there exists a constant $c'(\delta)>0$ such that for any integer lattice point $(n,m)$ in the first quadrant and satisfying $n<(6-\delta)m-c'(\delta)$ there exist infinitely many symplectic non-spin manifolds $Z_i, i\in \NN$ which have the following properties: the fundamental group $\pi_1(Z_i)=G, c_1^2(Z_i)=n, \chi_h(Z_i)=m;$  all of $Z_i$ are homeomorphic, but no two are diffeomorphic; all of them have negative Yamabe invariant and satisfy the strict generalized Hitchin-Thorpe Inequality, but none of them admits non-singular solutions to the normalized Ricci flow.
\end{main}

\begin{proof}
The proof is again based on the obstruction given in Corollary \ref{cor}. For constructing our examples we employ Gompf's techniques \cite{g} and use symplectic connected sum along symplectic sub-manifolds with trivial normal bundle. We need some standard blocks of symplectic manifolds.

The first block we denote by $X_G$ and it is the spin symplectic $4-$manifolds constructed by Gompf (see \cite{g} Theorem 6.2.). $X_G$ has fundamental group $G$, $c_1^2(X_G)=0$ and contains, as a symplectic submanifold, a $2-$torus of self-intersection $0.$ 

The second block, denoted by $E(n),$ is the family of simply connected proper elliptic complex surfaces, with no multiple fibers and Euler characteristic $c_2(E(n))=12n$. A generic fiber is a symplectic torus of self-intersection $0$ and it is a known fact that its complement is simply connected.

The third block is one of the symplectic manifolds constructed by Braungardt  and Kotschick, which has $\chi_h(X)=m.$ These manifolds are obtained as symplectic sums of simple symplectic manifolds along Riemann surfaces with trivial normal bundles. For example, one of of these manifolds, $S_{1,1}$ (example 2 in \cite{b-k}), has such  a symplectic sub-manifold of genus $2.$ Hence the manifold $X$ has a sympletic sub-manifold which is a Riemann surface of genus $2$ and has trivial canonical line bundle, which we denote by $R.$


The fourth block has a linking role. It is $E(4)$ with a special
 symplectic structure. This manifold has an important feature
 \cite[proof of Theorem $6.2$]{g}: it
 contains a torus and a genus $2$ Riemann surface as disjoint
 symplectic submanifolds. We denote them by $T$ and $F$
 respectively. Both $T$ and $F$ have self-intersection zero 
 and their complement $E(4)\setminus (F\cup T)$ is simply
 connected.

We are now ready to construct our symplectic manifolds. Let:
$$N=X_G~ ~\#_{T^2}~ E(4)~\#_{\Sigma_2}~X ~\#_{\Sigma_2}~E(4)~\#_{T^2}~E(2),$$
$$Z=N\#p\cpb,$$
where $\#_{T^2}$'s are the symplectic sums along tori of
 self-intersection zero and $\#_{\Sigma_2}$'s are fiber sums along Riemann surfaces of genus $2,$ represented by $F \subset E(4)$ and a copy of $R \subset X$.
  The last operation is simply a connected sum where
  $p=[\frac 13 (c_1^2(N)+2)]~>0.$

Gompf showed \cite{g} that the numerical invariants of a manifold obtained from symplectic sumation along a Riemann surface $\Sigma$ can be computed as
$$c_1^2(M\#_{\Sigma}M')=c_1^2(M)+c_1^2(M')-4\chi(\Sigma)$$
$$\chi_h(M\#_{\Sigma}M')=\chi_h(M)+\chi_h(M')-\frac12 \chi(\Sigma).$$
Hence the numerical invariants of $N$ are $c_1^2(N)=16+c_1^2(X), ~\chi_h(N)=12+\chi_h(X_G)+\chi
_h(X).$

The fundamental group of $Z$ can be easily computed by
 Seifert-Van Kampen Theorem to be $G$.

To obtain distinct differential structures on $Z$, we
 take logarithmic transformations of different multiplicities of the same parity
 along a generic fiber of $E(2)$. Then by the gluing formula for the
 Seiberg-Witten invariants \cite{SW-glue} (Cor 15,20) the
 manifolds have different Seiberg-Witten invariants. Hence we have
 constructed infinitely many non-diffeomorphic manifolds. We
 denote them by $Z_i$. 
 

 These manifolds are
 all homeomorphic. To show this, we can first do the
 logarithmic transformations on $E(2)$, this yields
 homeomorphic manifolds. By taking the fiber sum along a
 generic fiber with the remaining terms we
 obtain homeomorphic manifolds.

Corollary \ref{cor} implies that there are no non-singular solutions of the normalized Ricci flow on any $Z_i.$ 

The last thing that we need to verify is to find the possible pairs $(c_1^2(Z),\chi_h(Z))$ in the integer lattice.
We have that
$$c_1^2(Z)=c_1^2(N)-p=c_1^2(X)+16-[\frac 13 (c_1^2(X)+18)]=c_1^2(X)+10-[\frac 13(c_1^2(X)]$$
But  the manifold $X$ is obtained by the construction in Proposition \ref{geogr-sympl} so

\begin{align}
c_1^2(Z)& <  \frac23 c_1^2(X)+1+10\notag\\
 &<\frac 23 ((9-\eps)\chi_h(X)-c(\eps))+11\notag\\
 &=(6-\frac 23 \eps)(\chi_h(N)-\chi_h(X_G)-12)-\frac23c(\eps)+11 \notag\\
 &=(6-\frac 23 \eps)\chi_h(Z)-c'(\frac23\eps)\notag
\end{align}
where $c'(\frac23\eps)=\frac23c(\eps)+(6-\frac 23\eps)(\chi_h(X_G)+12)-11.$

As the manifolds $X$ constructed in Proposition \ref{geogr-sympl} fill in the region $0<n<(9-\eps)m-c(\eps),$ we can find a corresponding region for the the topological invariants of $Z.$
Hence for any $\delta=\frac 23\eps>0$ there exists a constant $c'(\delta)=\frac23 c(\frac 32\delta)+(6-\delta)(\chi_h(X_G)+12)-11$ such that any positive integer lattice point $(n,m)$ in the region $n<(6-\delta)m-c'(\delta)$ 
is obtained as $(c_1^2(Z),\chi_h(Z))$ of a manifold $Z$ as above.
\end{proof}

For spin manifolds the obstructions a little bit more restrictive, but we have a similar statement:
\begin{main}\label{main-spin}
For any finite group $G$, there is an infinite family of spin $4-$manifolds $X_i$ with fundamental group $\pi_1(X_i)=G$ such that each manifold  has infinitely many, distinct smooth structures for which the Yamabe invariant is negative and there is no non-singular solution to the normalized Ricci flow. All these manifolds  satisfy the strict generalized Hitchin-Thorpe Inequality.
\end{main}

\begin{proof}
The manifolds $M_i$ are constructed as connected sums and fiber sums of different blocks.

For the first block we start with $X_G$, the spin symplectic $4-$manifolds constructed by Gompf.
Taking the symplectic connected sum with the elliptic surface $E(2n)$ along arbitrary fibers ($\{pt\}\times T^2\cong F_0\subset E(2n)$) we obtain a new symplectic, spin manifold which we denote by $N_G(n)$. As the complement of the generic fiber $F_0\subset E(2n)$ is simply connected, the manifold $N_G(n)$ satisfies the following: $\pi_1(N_G(n))=G$,  and its topological invariants can be computed to be $c_1^2(N_G(n))=0,~c_2(N_G(n))=24k+24n.$ Since $G$ is finite, $b_1(N_G(n))=0$, and $b^+(N_G(n))=4(k+n)-1\equiv 3\mod 4.$

The second block is obtained from $E(2)$
 by performing a logarithmic transformation of order $2n+1$ on one
 non-singular elliptic fiber. We denote the new manifolds by $T_n$. All $T_n$ are simply connected spin manifolds
 with $b^+=3$ and $b^-=19$, hence they are all homeomorphic. Moreover,
 $T_n$ are K{\"a}hler manifolds and $c_1(T_n)=2n f$, where $ f$ is the
 multiple fiber introduced by the logarithmic transformation. Hence $\pm 2nf$ is a basic class, and its Seiberg-Witten invariant is $\pm1$.

The third block is a "small" spin manifold $X_0$ that was used by Gompf \cite{g} in the proof of Theorem 6.2 (see also the explicit construction in the proof of Theorem 1.6 in \cite{sio}). Its invariants are $c_2(X_0)=152, ~\tau(X_0)=-96,~ c_1^2(X_0)=16,~\mathrm{and}~b^+=27~(\equiv 3 \bmod4).$

We may now define our manifolds:
$$M_{i,j}=X_0~ \#~ N_G(i)~\#~ T_j.$$

For fixed $i$, the manifolds $M_{i,j}$ are all homeomorphic 
as we take connected sums of homeomorphic manifolds.
 We denote this homeomorphism type by $M_i$.

If we consider the basic classes of the Bauer-Furuta invariant,
 then both $a=c_1(X_0)+c_1( N_G(i) )+c_1(T_j)$ and
 $b=c_1(X_0)+c_1( N_G(i) )-c_1(T_j)$ are basic classes, and then
 $4j ~| ~(a-b).$ But any manifold has a finite number of basic classes which are a diffeomorphism invariant. As we let $j$ take infinitely many values, this will imply that $M_{i,j}$ represent infinitely many types of diffeomorphism classes.
 
By Theorem \ref{ricci-ob-2} these manifolds do not have non-singular solutions of the normalized Ricci flow, but they satisfy the strict generalized Hitchin-Thorpe Inequality.
\end{proof}

\begin{main}\label{non-ex-spin}
There exists an $n_0>0$ such that for any $d>n_0$ the following $4-$manifolds:
\begin{enumerate}
\item $X_{1,n}=d(n+5)K3 \# (d(n+7)-1)(S^2\times S^2)$,  
\item $X_{2,n}=d(2n+5)K3 \#  (d(2n+6)-1) (S^2\times S^2)$. 
\end{enumerate}
where $n$ is a positive integer, admit infinitely many non-equivalent free $\ZZ_{d}$-actions, such that there is no $\ZZ_d-$equivariant non-singular solution to the normalized Ricci flow on $X_{1,n},X_{2,n}$ for any of the $\ZZ_d$ actions.
\end{main}

\begin{proof}
We begin by constructing a simply connected, spin manifolds with small topological invariants and $b^+\equiv3\bmod 4.$ One such manifold is given by a smooth hypersurface of tridegree $(4,4,2)$ in $\bcp^1\times\bcp^1\times\bcp^1,$ which we denote by $X.$ Its numerical invariants can be easily computed to be $c_1^2(X)=16,~ c_2(X)=104,~b^+(X)=19$ and it is simply connected. Then by Freedman's Theorem, $X$ is homeomorphic to $4K3\#7(S^2\times S^2).$ A result of Wall \cite{wall} tells us that there exists an integer $n_0$ such that $X\# n_0(S^2\times S^2)$ becomes diffeomorphic to $4K3\#7(S^2\times S^2)\#n_0(S^2\times S^2).$

Assuming the notation from the previous theorem, let:
$$M_{1,n}^j= X\# T_j\# E(2n)\#Y_d$$
$$M_{2,n}^j= X\#E(2)\# T_j\# E(2(2n-1))\#Y_d$$
By Theorem \ref{ricci-ob-2} the above manifolds do not admit an Einstein metric.
 
The manifolds $\{M_{1,n}^j~|~j\in \NN\}$ are all homeomorphic. However, as argued in the proof of the previous  theorem they represent
 infinitely many differential structures. Moreover, if we consider 
 the universal cover, $\wt{M_{1,n}^j},$ it is diffeomorphic to 
 $dX\# dT_j\# dE(2n)\#(d-1)(S^2\times S^2).$ 
 But Mandelbaum \cite{man}
  proved that both $T_j$ and $E(2n)$ completely decompose as connected 
  sums of $K3'$s and $S^2\times S^2$'s after taking the connected sum 
  with one copy of $S^2\times S^2.$ Hence, for $d>n_0,$ the manifold
  $ \wt{M_{1,n}^j}$ is diffeomorphic to 
 $d(4K3\#7(S^2\times S^2))\# d(K3)\#d(nK3\#(n-1)
 (S^2\times S^2))\#(d-1)(S^2\times S^2)$ i.e. to
 $d(n+5)K3\#(d(n+7)-1)(S^2\times S^2)=X_{1,n}.$ 
 
Notice that the diffeomorphism type of the universal cover does not depend on $j$. Hence on $X_{1,n}$ we have constructed infinitely many non-equivalent, free actions of $\ZZ_d,$ such that there is no non-singular $\ZZ_d-$equivariant solution of the normalized Ricci flow. But all $M_{1,n}$ satisfy the strict generalized Hitchin-Thorpe Inequality.
 
 Redoing the same arguments for the second example $M_{2,n}^j,$ gives us the results for the second family of manifolds.
\end{proof}

\begin{rmk}
The obstruction and the results in this paper are on manifolds with $b^+\geq 2$. A generalization of this obstructions for $b^+=1$ will be the subject of a new paper \cite{irs}, which is currently in preparation.
\end{rmk}

\vfill

{\footnotesize 
\noindent
{Masashi Ishida, \\
{Department of Mathematics,  
Sophia University, \\ 7-1 Kioi-Cho, Chiyoda-Ku, 
 Tokyo 102-8554, Japan }\\
{\sc e-mail}: ishida@mm.sophia.ac.jp}

\vspace{0.1cm}

{\footnotesize 
\noindent
{Ioana {\c S}uvaina, \\
{IHES, 35 route de Chartres, 91440 Bures-sur-Yvette, France}\\
{\sc e-mail}: ioana@ihes.fr}

\end{document}